\documentclass[12pt]{article}

\usepackage{amssymb}
\usepackage{amsthm}

\usepackage{amsmath}
\usepackage{authblk}

\newtheorem{defn}{Definiton}
\newtheorem{thm}{Theorem}
\newtheorem{exmp}{Example}
\newtheorem{prop}{Proposition}
\newtheorem{cor}{Corollary}
\newtheorem{lem}{Lemma}

\begin{document}



\title{\Large On a generalization of compact and connected spaces}
\author{Neboj\v sa Elez\thanks {University of East Sarajevo {\tt nebojsa.elez@ff.ues.rs.ba}}
\and {Ognjen Papaz\thanks{University of East Sarajevo {\tt ognjen.papaz@ff.ues.rs.ba}}}}
\date{}

\maketitle

\begin{abstract}
In this paper we will give two different natural generalizations of compact spaces and connected spaces simultaneously. We will show that these generalizations coincide for the subspaces of the real line and that they  differ for subspaces of $\mathbb R^2$.
\end{abstract}








\section{Introduction}
\label{}




Classes of compact and connected spaces, although not very related, share some similar properties. They are closed with respect to taking continuous image and topological product. 

Intersection of class of compact and class of connected spaces, the class of continua, is well known. Class of continua is closed with respect to taking continuous image and topological product. 

Union of class of compact and class of connected spaces is not particularly interesting. 

The smallest class of spaces that contains both compact and connected spaces and that is closed with respect to taking continuous image and topological product is class of spaces that are continuous image of product of compact and connected space. The following subspace of the real line is an example of such space
\[\{0\}\cup \bigcup_{n=1}^{+\infty}\left(\frac{1}{n+1},\frac{1}{n}\right).\] 
This space has infinitely many connected components but it doesn't have infinitely many open connected components. This property can inspire the following definition.
 
\begin{defn} We say that a topological space is a GCC space (Generalized compact and connected space) if it cannot be represented as a union of infinitely many of its disjoint non-empty open subsets.
\end{defn}
In the first part of the paper we will give characterizations of the class of GCC spaces. 

It's easy to see that continuous image of a GCC space is a GCC space. It turns out that product of two GCC spaces doesn't have to be a GCC space. 

Similar class of spaces, having property that every clopen cover of a space has finite subcover, was introduced in \cite{sterp}. 

In the second part we will consider the following subclass of GCC spaces that is closed with respect to taking continuous image and topological product.

\begin{defn} We say that a topological space is a CCC space if it is a \textbf{c}ontinuous image of a product of \textbf{c}ompact and \textbf{c}onnected space.
\end{defn}

As we have mentioned earlier, the class of CCC spaces is a minimal class of spaces that contains all compact and connected spaces and that is closed with respect to product and taking continuous image.

We will give intrinsic property that characterizes CCC spaces. Namely, we will show that a space $X$ is a CCC space if and only if it has compact subspace that intersects every connected component of $X$.

In the last part of the paper we will show that these two classes are equal on the real line, and we will give an example of a subspace of $\mathbb R^2$ which is a GCC space but not a CCC space.

\section{Main results}

\begin{thm} For a topological space $X$, the following conditions are equivalent:\\

(1) $X$ is a GCC space;

(2) Every disjoint open cover of $X$ has a finite subcover;

(3) Every countable clopen cover of $X$ has a finite subcover;

(4) Every decreasing sequence of non-empty clopen subsets of $X$ has a non-empty intersection;

(5) There is no continuous function from $X$ onto $\mathbb{N}$.

(6) $X$ is not a direct sum of infinitely many of its nonempty subspaces.
\end{thm}

\begin{proof} $(1)\Rightarrow (2)$. Let $(U_i)_{i\in I}$ be a disjoint open cover of $X$. Since $X$ is GCC space, it follows that all but finitely many of the sets $U_i$ are empty. Hence, $(U_i)_{i\in I}$ has finite subcover.

$(2)\Rightarrow (3)$. Let $(U_n)_{n\in\mathbb{N}}$ be a clopen cover of $X$. For every $n\in\mathbb{N}$ let $V_n=U_n\setminus\bigcup_{k<n}U_k$. Family $(V_n)_{n\in\mathbb{N}}$ is a disjoint open cover of $X$, hence it has a finite subcover. Since $(V_n)_{n\in\mathbb{N}}$ is a refinement of the cover $(U_n)_{n\in\mathbb{N}}$, it follows that $(U_n)_{n\in\mathbb{N}}$ also has a finite subcover.

$(3)\Rightarrow (4)$. Let $\{F_n\}_{n=1}^{+\infty}$ be a decreasing sequence of non-empty clopen subsets of $X$. If $\bigcap_{n=1}^{+\infty}F_n=\emptyset$, then the family $(X\setminus F_n)_{n\in\mathbb{N}}$ is a countable clopen cover of $X$ that has no finite subcover, which is impossible by the assumption of the theorem. Therefore, $\bigcap_{n=1}^{+\infty}F_n\neq\emptyset$. 

$(4)\Rightarrow (5)$. Assume to the contrary that there exists a continuous function $f$ from $X$ onto $\mathbb{N}$. For every $n\in\mathbb{N}$ let $F_n=\mathbb{N}\setminus\{1,2,\ldots,n\}$. Since each of the sets $F_n$ is clopen in $\mathbb{N}$, it follows that $\{f^{-1}(F_n)\}_{n=1}^{+\infty}$ is a decreasing sequence of non-empty clopen subsets of $X$. But, $\bigcap_{n=1}^{+\infty}f^{-1}(F_n)=f^{-1}(\bigcap_{n=1}^{+\infty}F_n)=f^{-1}(\emptyset)=\emptyset$, which is a contradiction. 

$(5)\Rightarrow (6)$. Let us assume that $X$ is the direct sum of the sequence $\{X_n\}_{n=1}^{+\infty}$ of disjoint non-empty subspaces of $X$. The function $f:X\to\mathbb{N}$, defined with $f(X_n)=\{n\}$, is a continuous function from $X$ onto $\mathbb{N}$. Thus, we have a contradiction. 

$(6)\Rightarrow (1)$. This implication is obviously true.
\end{proof}

One readily proves the following proposition.

\begin{prop} Each of the following statements holds.\\

1. All countably compact, compact and connected spaces are GCC spaces. 

2. An image of a GCC space under a continuous mapping is a GCC space. 

3. A clopen subspace of a GCC space is a GCC space. 

4. A space $X$ is a GCC space if it has a dense GCC subspace. 

5. A space $X$ is a GCC space if it can be represented as a union of finitely many of its GCC subspaces. 

6. If $(X,\tau)$ is a GCC space and if $\sigma\subseteq\tau$, then $(X,\sigma)$ is also a $GCC$ space.
\end{prop}

The CLP-compact spaces, that were introduced in \cite{sterp}, are defined by having the property that every clopen cover has finite subcover. Every CLP-compact space is a GCC space.  In the realm of second countable spaces every GCC space is CLP-compact. The following example shows that class of GCC spaces is strictly larger then the class of CLP-compact spaces.

\begin{exmp} First uncountable ordinal $\omega_1$ is a GCC space that is not CLP-compact. Indeed, $\omega_1$ is countably compact so it is a GCC space. And $\omega_1$ is not CLP-compact since clopen cover of $\omega_1$ consisting of the sets of the form $[0,a], a<\omega_1$, has no finite subcover.
\end{exmp}

The following example shows that product of two GCC spaces doesn't need to be a GCC space.

\begin{exmp} Famous example of Novak \cite[3.10.19 Example]{eng} exhibits two countably compact subspaces $X$ and $Y$ of $\beta\mathbb N$ whose product $X\times Y$ contains $\{(1,1),(2,2),\ldots\}$ as an open, closed and discrete subset. Thus, $X$ and $Y$ are GCC spaces and their product is not a GCC space.
\end{exmp}

We will now consider the class of $CCC$ spaces.

\begin{prop} Every CCC space is a GCC space.
\begin{proof}

Let $X$ be a CCC space and let $X=f(K\times C)$, where $f$ is continuous, $K$ is compact and $C$ is connected. To prove that $X$ is a GCC space it is enough to prove that $K\times C$ is a GCC space. Let $(U_i)_{i\in I}$ be a disjoint open cover of $K\times C$. We pick some element $c\in C$. Subspace $K\times\{c\}$ of $K\times C$ is compact so there is a finite subfamily $\{U_{i_1},U_{i_2},\ldots,U_{i_n}\}$ of $(U_i)_{i\in I}$ that covers $K\times\{c\}$. We will prove that $\{U_{i_1},U_{i_2},\ldots,U_{i_n}\}$ covers $K\times C$. Let $k\in K$. Subspace $\{k\}\times C$ of $K\times C$ is connected and there is some $U_{i_k}\in\{U_{i_1},U_{i_2},\ldots,U_{i_n}\}$ such that $(\{k\}\times C)\cap U_{i_k}\neq\emptyset$. Since $U_{i_k}$ is clopen set we have that $\{k\}\times C\subseteq U_{i_k}$. 
\end{proof}
\end{prop}

\begin{prop} The class of CCC spaces is closed with respect to product and taking continuous image.

\begin{proof} One easily sees that a continuous image of a CCC space is a CCC space. We prove that the class of CCC spaces is closed with respect to product. Let $(X_i)_{i\in I}$ be a family of CCC spaces and let each $X_i=f_i(K_i\times C_i)$ where $f_i$ is continuous, $K_i$ is compact, and $C_i$ is connected. We define a mapping 
\[f:\prod_{i\in I}(K_i\times C_i)\to\prod_{i\in I}X_i\]
in the following way
\[f(((k_i,c_i)_{i\in I}))=(f_i(k_i,c_i))_{i\in I}.\]
The mapping $f$ is a continuous surjection and since 
\[\prod_{i\in I}(K_i\times C_i)\cong\left(\prod_{i\in I}K_i\right)\times\left(\prod_{i\in I}C_i\right),\]
we have that $\prod_{i\in I}X_i$ is a CCC space.
\end{proof}
\end{prop}

Now we give an intrinsic characterization of $CCC$ spaces. 

\begin{thm}

A space $X$ is a CCC space if and only if $X$ has a compact subspace that intersects every connected component of $X$. 

\begin{proof} Let $X$ be a CCC space $X=f(K\times C)$ where $f$ is continuous, $K$ is compact and $C$ is connected. We pick some element $c\in C$. Then $K\times\{c\}$ is a compact subspace of $K\times C$ that intersects every connected component of $K\times C$ and $f(K\times\{c\})$ is a compact subspace of $X$ that intersects every connected component of $X$.

Now we prove the converse. Let $K$ be a compact subspace of $X$ that intersects every connected component of $X$. 

For every $k\in K$ let $C_k$ be connected component of $X$ that contains $k$. We will regard each $C_k$ as a based space $(C_k,k)$. Let
\[C=\bigvee_{k\in K} C_k\]
be the wedge sum of the spaces $(C_k,k)$. The space $C$ is connected because every $C_k$ is connected. 

Let $A(K)=K\cup (K\times\{1\})$ be Alexandroff duplicate of $K$ (see \cite{ad}). The family 
\[\{U\cup((U\setminus\{x\})\times\{1\}):U\ \text{is open in}\ K, x\in U\}\cup\{(x,1):x\in K\}\]
is subbasis of topology of $A(K)$.
Since $K$ is compact space, $A(K)$ is also compact. We will construct a continuous surjection 
\[f:A(K)\times C\to X.\]

For every $k,k'\in K$ we define $g_{k,k'}:C_{k'}\to X$ in the following way
\[g_{k,k'}(x)=
\begin{cases}
x, k=k'\\
k, k'\neq k.
\end{cases}\]
Each $g_{k,k'}$ is continuous map and sends base point of $C_{k'}$ to $k$. Thus, there exists unique continuous map 
\[g_k:C\to X,\]
such that $ g_k\circ i_{k'}=g_{k,k'}$ for every $k'$, where $i_{k'}:C_{k'}\to C$ is a natural embedding.

We now define $f$. Let $k\in K$. We define
\[f_{|\{(k,1)\}\times C}=g_k,\]
and 
\[f_{|\{k\}\times C}=\text{constant map with the value}\ k.\]
Its easy to see that $f$ is a surjection. Now we prove the continuity of $f$. Since each of the sets $\{(k,1)\}\times C$ is open in $A(K)\times C$ and each $g_k$ is continuous we only need to check the continuity of $f$ at an arbitrary point $(k,c)\in K\times C$. 

Let $k_0\in K$ and $c\in i_{k_1}(C_{k_1})$. We have $f(k_0,c)=k_0$. Let $V$ be an open neighborhood of $k_0$ in $X$. For every $k\in K$ let 
\[U_k=
\begin{cases}
i_k(V\cap C_k), k\in V\setminus \{k_1\},\\
i_k(C_k), k\in (K\setminus V)\cup\{k_1\}.
\end{cases}
\]
Let $U=\bigcup_{k\in K}U_k$. We have that $U$ is an open neighborhood od $c$ in $C$. 

For every $k\in G=V\setminus \{k_1\}$ we have
\[g_k(U_{k'})=g_k(i_{k'}(V\cap C_{k'}))=g_{k,k'}(V\cap C_{k'})\subseteq V,\]
for $k'\in G$, and
\[g_k(U_{k'})=g_k(i_{k'}(C_{k'}))=g_{k,k'}(C_{k'})=\{k\},\]
for $k'\in (K\setminus V)\cup\{k_1\}$.

Thus, for every $k\in G$ we have $g_k(U)\subseteq V$.

Let 
\[V'=V\cup (G\times{1}).\]
We have that $V'$ is an open neighborhood of $k_0$ in $A(K)$.

Now, we have that
\begin{align*}
f(V'\times U)&=f((V\times U)\cup ((G\times{1})\times U))=\\
&f(V\times U)\cup f((G\times\{1\})\times U)=\\
&V\cup \bigcup_{(k,1)\in G\times\{1\}}f((k,1)\times U)=\\
&V\cup \bigcup_{k\in G}g_k(U)\subseteq V.
\end{align*}

\end{proof}

\end{thm}

The following example is an example of a subspace of $\mathbb R^2$ which is a GCC space but not a CCC space.

\begin{exmp}
For every  $n\in\mathbb N$ let $A_n=\bigcup_{m\geq n+1}([0,1]\times\{1/n+1/m\})$ and $x_n=(1,1/n)$. Let
\[X=\{x_n:n\in\mathbb N\}\cup\left(\bigcup_{n\in\mathbb N}A_{n}\right)\cup\{\mathbf{0}\}.\]

We will first show that $X$ is not a CCC space. If $K$ is compact subspace of $X$ that intersects every connected component of $X$ then $\{x_n:n\in\mathbb N\}\subseteq K$. But $\{x_n:n\in\mathbb N\}$ is infinite closed and discrete subspace of $X$, and we have a contradiction.  

We will now prove that $X$ is a GCC space. Let $(U_i)_{i\in i}$ be a clopen cover of $X$. Let $\mathbf{0}\in U_{i_0}$. There exists $n_0\in\mathbb N$ such that $A_n\subseteq U_{i_0}$ for all $n>n_0$. We have that $x_n\in\overline{A_n}$ for every $n\in\mathbb N$, thus $x_n\in U_{i_0}$ for all $n>n_0$. Let
\[x_1\in U_{i_1},x_2\in U_{i_2},\ldots,x_{n_0}\in U_{i_{n_0}}.\]
Each of the sets $U_{i_k}$ contains all but finitely many connected components of $A_k$. Thus $\bigcup_{k=0}^{n_0}U_{i_k}$ contains all but finitely many connected components of $X$. From here we can conclude that $(U_i)_{i\in I}$ has finite subcover of $X$.
\end{exmp}

We note that every connected component of the space $X$ from the previous example is a quasicomponent. So, $X$ is also an example of a GCC subspace of $\mathbb R^2$ that doesn't have compact subspace that intersects every quasicomponent of $X$.\\

Now we restrict our attention to the subspaces of $\mathbb R$.

\begin{thm} A space $X\subseteq\mathbb{R}$ is a GCC space if and only if every monotone sequence $\{x_n\}_{n=1}^{+\infty}$ of real numbers, such that
\[x_{2n}\in X\ \text{and}\ x_{2n-1}\in\mathbb{R}\setminus X\ \text{for every}\ n\in\mathbb{N},\]
is convergent and $\lim_{n\to+\infty}x_n\in X$.

\begin{proof} Let $X$ be a GCC space and let $\{x_n\}_{n=1}^{+\infty}$ be a monotone sequence of real numbers such that $x_{2n}\in X$ and $x_{2n-1}\in\mathbb{R}\setminus X$ for every $n\in\mathbb{N}$. Let us first prove that the $\lim_{n\to+\infty}x_n$ exists. If $\lim_{n\to+\infty}x_n$ does not exist then the sequence $\{x_n\}_{n=1}^{+\infty}$ is unbounded, and the family
\[\{(-\infty,x_1)\cap X\}\cup\{(x_{2n-1},x_{2n+1})\cap X:n\in\mathbb{N}\},\]
i.e. the family
\[\{(x_{2n-1},x_{2n+1})\cap X:n\in\mathbb{N}\}\cup\{(x_1,+\infty)\cap X\},\]
depending on whether the sequence $\{x_n\}_{n=1}^{+\infty}$ is increasing or decreasing, is a disjoint open cover of $X$ that has no finite subcover, which cannot exist because $X$ is a $GCC$ space. Thus, the $\lim_{n\to+\infty}x_n$ exists. Let $x=\lim_{n\to+\infty}x_n$. If $x\notin X$, then the family
\[\{(-\infty,x_1)\cap X\}\cup\{(x_{2n-1},x_{2n+1})\cap X:n\in\mathbb{N}\}\cup\{(x,+\infty)\cap X\},\]
i.e. the family
\[\{(-\infty,x)\cap X\}\cup\{(x_{2n-1},x_{2n+1})\cap X:n\in\mathbb{N}\}\cup\{(x_1,+\infty)\cap X\},\]
depending again on whether the sequence $\{x_n\}_{n=1}^{+\infty}$ is increasing or decreasing, is a disjoint open cover of $X$ that has no finite subcover. 

Now let us prove the converse. If $X$ is not a GCC space then there is a sequence $\{U_n\}_{n=1}^{+\infty}$ of disjoint non-empty open subsets of $X$ such that $X=\bigcup_{n=1}^{+\infty}U_n$. Let $\{a_n\}_{n=1}^{+\infty}$ be a sequence in $X$ such that $a_n\in U_n$ for every $n\in\mathbb{N}$ and let $\{a_{n_k}\}_{k=1}^{+\infty}$ be a monotone subsequence of the sequence $\{a_n\}_{n=1}^{+\infty}$. We will assume that $\{a_{n_k}\}_{k=1}^{+\infty}$ is increasing. For every $k\in\mathbb{N}$ we have $(a_{n_k},a_{n_{k+1}})\setminus X\neq\emptyset$, so there is an increasing sequence $\{b_n\}_{n=1}^{+\infty}$ such that $b_{2k}=a_{n_k}$ and $b_{2k+1}\in(a_{n_k},a_{n_{k+1}})\setminus X\subseteq\mathbb{R}\setminus X$ for every $k\in\mathbb{N}$. By the assumption of the theorem we have $\lim_{n\to+\infty}b_n\in X$, but that is impossible because every set $U_n$ contains at most one of the members of the sequence $\{b_n\}_{n=1}^{+\infty}$. We analogously reach a contradiction if $\{a_{n_k}\}_{k=1}^{+\infty}$ is decreasing. 
\end{proof}
\end{thm}

\begin{cor} If $X$ is a GCC subspace of $\mathbb{R}$ and if $(C_i)_{i\in I}$ is a family of components of the space $X$, then $\overline{X}=\bigcup_{i\in I}\overline{C_i}$.
\end{cor}

\begin{proof}
Let $x\in\overline{X}\setminus X$ and let $\{x_n\}_{n=1}^{+\infty}$ be a strictly increasing sequence in $X$ such that $\lim_{n\to+\infty}x_n=x$. Suppose that $x\notin\bigcup_{i\in I}\overline{C_i}$. There exists a subsequence $\{x_{n_k}\}_{k=1}^{+\infty}$ of $\{x_n\}_{n=1}^{+\infty}$ and a sequence $\{y_k\}_{k=1}^{+\infty}$ such that 
\[y_k\in (x_{n_k},x_{n_{k+1}})\setminus X\]
for every $k\in\mathbb N$. Now, the sequence $\{z_k\}_{k=1}^{+\infty}$, defined with
\[z_{2k+1}=x_{n_k}\ \text{and}\ z_{2k}=y_k,\]
contradicts theorem 3.
\end{proof}

Similar argument can be used to prove the following corollaries.

\begin{cor} If $X$ is a GCC subspace of $\mathbb R$ then $\mathbb R\setminus X$ is locally connected. If $X$ is bounded subspace of $\mathbb R$ such that $X\setminus\mathbb R$ is locally connected then $X$ is a GCC space.
\end{cor}

\begin{cor} Let $X$ be a GCC space and let $f$ be a continuous mapping form $X$ to $\mathbb{R}$. If $a=\sup f(X) (\inf f(X))\in\mathbb R$, then $a\in f(X)$ or $(a-\epsilon,a) ((a,a+\epsilon))\subseteq f(X)$ for some $\epsilon>0$. If $f$ is unbounded, then $f(X)$ contains an unbounded interval.
\end{cor}

We will now prove that for every GCC subspace $X$ of $\mathbb R$ is continuous image of the product $C\times\mathbb R$, where $C$ is Cantor set. We begin with the following lemma.\\

\begin{lem} Let $X$ be a subspace of $\mathbb{R}$ and let $(C_i)_{i\in I}$ be the family of all connected components of $X$. Let $\varphi_X:(C_i)_{i\in I}\to 2^\mathbb R$ be a mapping that satisfies the following two conditions:\\
\indent 1) $C_i\setminus C_i^\circ\subseteq \varphi_X(C_i)\subseteq C_i$;\\
\indent 2) $C_i^\circ\neq\emptyset\Rightarrow \varphi_X(C_i)\cap C_i^\circ$ is a singleton.\\
Then $X$ is a GCC space if and only if $\bigcup_{i\in I}\varphi_X(C_i)$ is compact subspace of $\mathbb R$.

\begin{proof}  Let $Y=\bigcup_{i\in I} \varphi_X(C_i)$. If $Y$ is compact then $X$ is a GCC space because $Y$ intersects every component of $X$, hence $X$ is a GCC space.

Let $X$ be a GCC space. To prove that $Y$ is compact it is enough to prove that every strictly monotone sequence in $Y$ converges. Let $\{y_n\}_{n=1}^{+\infty}$ be a strictly increasing sequence in $Y$. Let $x_{2n}=y_{4n}$ for every $n\in\mathbb N$. We have that $y_{4n}\neq y_{4n+4}$. Since every $\varphi_X(C_i)$ contains at most 3 points we have that
\[C_{x_{2n}}=C_{y_{4n}}\neq C_{y_{4n+4}}=C_{x_{2n+2}},\]
where $C_x$ is a connected component of $x$ in $X$. Since elements $x_{2n}$ and $x_{2n+2}$ belong to different connected components of $X$ we have that $(x_{2n},x_{2n+2})\setminus X\neq\emptyset$. For every $n\in\mathbb N$ we pick $x_{2n+1}\in(x_{2n},x_{2n+2})\setminus X$. Now, by the theorem 3, we have that $\{x_n\}_{n=1}^{+\infty}$ converges to a point $x\in X$ and
\[\lim y_n=\lim y_{4n}=\lim x_{2n}=x.\]
\end{proof}
\end{lem}

\begin{thm} Every GCC subspace of $\mathbb{R}$ is a continuous image of the space $C\times\mathbb{R}$, where $C$ is the Cantor set. 

\begin{proof} Let $X$ be a GCC subspace of $\mathbb{R}$ and let $(C_i)_{i\in I}$ be the family of connected components of $X$. Let $\varphi_X$ be a mapping as in the previous lemma and let $A=\bigcup_{i\in I}\varphi_X(C_i)$. The space $A$ is a compact subspace of $\mathbb{R}$ by the previous lemma. Since every compact subspace of $\mathbb{R}$ is a continuous image of the Cantor space it suffices to prove that $X$ is a continuous image of the space $A\times\mathbb{R}$. Let $f$ be a mapping from $A\times\mathbb{R}$ to $X$ defined by the following conditions:\\

(1) if $x\in \varphi_X(C_i)\cap C_i^\circ$ for some $i\in I$, then $f_{|\{x\}\times\mathbb{R}}(x,y)=g(y)$, for some continuous mapping $g$ from $\mathbb{R}$ onto $C_i$;

(2) if $x\in C_i\setminus C_i^\circ$ for some $i\in I$, then $f_{|\{x\}\times\mathbb{R}}(x,y)=x$ for every $y\in\mathbb{R}$.\\

It is easy to see that $f$ is a surjection. 

Let us prove that $f$ is continuous. Let $(x,y)\in A\times\mathbb{R}$. 

If $x\in \varphi_X(C_i)\cap C_i^\circ$ for some $i\in I$, then $x$ is an isolated point in $A$ and $f$ is continuous at $(x,y)$ because $\{x\}\times\mathbb{R}$ is an open neighborhood of $(x,y)$ in $A\times\mathbb{R}$ and $f_{|\{x\}\times\mathbb{R}}$ is continuous. 

Now let $x\in C_{i_0}\setminus C_{i_0}^\circ$ for some $i_0\in I$. Then we have that $f(x,y)=x$. Let $\epsilon>0$. 

Let 
\[J=\{j\in I\setminus\{i_0\}: x-\epsilon\in C_j\vee x+\epsilon\in C_j\}.\]
Let 
\[U=(x-\epsilon,x+\epsilon)\setminus\left((\varphi_X(C_{i_0})\cap C^\circ_{i_0})\cup\bigcup_{j\in J}C_j\right).\]
We note that the set $U$ is open since $|J|\leq 2$, and we also note that $x\in U$. 
Further, for every component $C_i\ (i\neq i_0)$ of $X$ we have 
\[C_i\cap U\neq\emptyset\Rightarrow C_i\subseteq U,\]
and we have that
\[f((U\cap \varphi_X(C_{i_0}))\times\mathbb R)\subseteq U.\]
Let $V=U\cap A$. We have that $V\times\mathbb R$ is an open neighborhood of $(x,y)$ in $A\times\mathbb R$ and that 
\[f(V\times\mathbb R)\subseteq U\subseteq (x-\epsilon,x+\epsilon).\]
Thus, $f$ is continuous at $(x,y)$.
\end{proof}
\end{thm}

\end{document}